\documentclass[11pt]{amsart}

\usepackage[USenglish]{babel}
\usepackage[T1]{fontenc} 
\usepackage[utf8,latin1]{inputenc}

\usepackage{amsmath,amsthm,amssymb,amsfonts}
\usepackage{dsfont}
\usepackage{mathtools}
\usepackage{nicematrix}
\usepackage{mathrsfs}
\usepackage{bm}

\usepackage[font={small}]{caption}

\mathtoolsset{showonlyrefs}  

\usepackage{enumitem}
\usepackage{lmodern}
\usepackage[overload]{empheq}

\usepackage[bottom=3cm, top=3cm, left=3.5cm, right=3.5cm]{geometry}

\usepackage{subfigure} 

\usepackage[bookmarks=true]{hyperref}
\usepackage{xcolor}
\hypersetup{
    colorlinks,
    linkcolor={red!50!black},
    citecolor={blue!50!black},
    urlcolor={blue!80!black},
}

\usepackage{todonotes}

\newcommand{\R}{\mathbb{R}}

\newcommand{\sr}{sub-Riemannian }

\newcommand{\D}{\mathcal{D}}

\newcommand{\A}{\mathcal{A}}

\newcommand{\diverg}{\mathrm{div}_\omega}

\newcommand{\deltaS}{\delta_\mathrm{S}}
\newcommand{\smallpar}{r_0}
\newcommand{\eps}{\varepsilon}
\newcommand{\Spar}{S_{\smallpar}}
\newcommand{\muS}{\mu_\mathrm{S}}

\newcommand{\spn}{\mathrm{span}}

\theoremstyle{plain}
\newtheorem{thm}{Theorem}[section]

\newtheorem{lem}[thm]{Lemma}
\newtheorem*{lem*}{Lemma}
\newtheorem{prop}[thm]{Proposition}

\theoremstyle{definition}
\newtheorem{defn}[thm]{Definition}

\theoremstyle{remark}
\newtheorem{rmk}[thm]{Remark}

\title{The Relative Heat Content for Submanifolds in Sub-Riemannian Geometry}
\date{\today}

\author[Tommaso Rossi]{Tommaso Rossi}
\address{Institut f\"ur Angewandte Mathematik, Universit\"at Bonn, Bonn, Germany}
\email{\href{mailto:rossi@iam.uni-bonn.de}{rossi@iam.uni-bonn.de}}

\begin{document}

\begin{abstract}
We study the small-time asymptotics of the relative heat content for submanifolds in sub-Riemannian geometry. First, we prove the existence of a smooth tubular neighborhood for submanifolds of any codimension, assuming they do not have characteristic points. Next, we propose a definition of relative heat content for submanifolds of codimension $k\geq 1$ and we build an approximation of this quantity, via smooth tubular neighborhoods. Finally, we show that this approximation fails to recover the asymptotic expansion of the relative heat content of the submanifold, by studying an explicit example.
\end{abstract}

\maketitle 
\tableofcontents

\section{Introduction}
In this paper, we study the small-time behavior of the relative heat content for submanifolds in sub-Riemannian geometry. 
Loosely speaking, the relative heat content for a submanifold $S\subset M$ can be regarded as the total amount of heat contained in $S$ at time $t$, corresponding to a uniform initial temperature distribution concentrated on $S$. When $S=\Omega$ is an open and bounded domain, the problem of finding an asymptotic expansion for the heat content has been extensively studied, see for example \cite{vdB-LG,vdB-G1,Savo-heat-cont-asymp,vdB-D-G,D-G,MR3358065} for the Euclidean and Riemannian case, and \cite{TW-heat-cont-hei,MR4223354,ARR-relative} for the \sr case. In particular, in this situation, in both the Riemannian and \sr setting, the small-time behavior of the heat content associated with $\Omega$ encodes geometrical information of $\partial\Omega$, such as its perimeter or its mean curvature (c.f.\ Theorem \ref{t:rel}). Thus, we may expect that, in an analogous way, the asymptotics of the relative heat content for a submanifold of higher codimension detect its geometrical invariants.  

Up to our knowledge, the relative heat content for submanifolds has never been systematically studied, not even in Riemannian geometry, so we propose the following definition. Let $M$ be a \sr manifold, equipped with a smooth measure $\omega$, and let $S\subset M$ be a smooth, compact submanifold of codimension $k\geq 0$. Let $\omega$ be a smooth measure on $M$, and let $\mu$ be a probability measure on $S$. Then, we consider $\mu$ as the initial datum for the heat equation in the sense of distributions, and study the associated Cauchy problem:
\begin{equation}\label{eq:intro_he_submnf}
\begin{aligned}
(\partial_t -\Delta)u(t,x)  =  0, & &\qquad &\forall (t,x) \in (0,\infty) \times M, \\
u(t,\cdot) \xrightarrow{t\to0}  \mu, & &  \qquad & \text{in }\mathcal{D}'(M),
\end{aligned}
\end{equation}
where $\Delta=\diverg\circ\nabla$ is the usual sub-Laplacian associated with $\omega$. A solution to this problem, in the sense of distribution, is given by 
\begin{equation}
\label{eq:intro_sol_S}
u(t,x)=\int_Sp_t(x,y)d\mu(x), \qquad\forall\, (t,x) \in (0,\infty) \times M,
\end{equation}
where $p_t(x,y)$ is the usual heat kernel associated with $\Delta$ and $\omega$. We define the \emph{relative heat content for a submanifold} $S$ as:
\begin{equation}
\label{eq:intro_hc_S}
H_S(t)=\int_S\int_Sp_t(x,y)d\mu(x)d\mu(y),\qquad\forall\,t>0.
\end{equation}
Notice that, on the one hand if $S=\{x_0\}$ and $\mu=\delta_{x_0}$, we obtain the trace heat kernel $p_t(x_0,x_0)$. On the other hand, if $S$ is a $0$-codimensional submanifold, i.e. $S$ is an open, relatively compact and smooth set, and we choose $\mu=\mathds{1}_S\omega$ , then \eqref{eq:intro_hc_S} coincides with the usual relative heat content as defined in \cite{ARR-relative}. Therefore, a small-time asymptotic expansion of \eqref{eq:intro_hc_S} would include many cases of interest, ranging from the trace heat kernel asymptotics, see for example \cite{MR3207127,YHT-2}, to the results of contained in \cite{ARR-relative}.

Our attempt to compute the asymptotics of \eqref{eq:intro_hc_S} consists in building a suitable approximation of it, using smooth tubular neighborhoods of $S$. When $S$ is a smooth hypersurface, the existence of a tubular neighborhood is guaranteed assuming that there are no characteristic points, cf.\ \cite[Def. 2.4]{ARR-relative}. Therefore, as a first step, we introduce a definition of a non-characteristic submanifold, cf.\ Definition \ref{def:nonchar_submnf}, generalizing the classical one of non-characteristic hypersurface. Then, denoting by $\delta_S\colon M\rightarrow [0,+\infty)$ the distance function from $S$, see \eqref{eq:dist_S} for the precise definition, we prove the following.

\begin{thm}
\label{thm:intro_tub_neigh}
Let $M$ be a \sr manifold and $S\subset M$ be a compact smooth non-characteristic submanifold of codimension $k\geq 1$. Then, there exists $\smallpar>0$ such that, denoting by $\Spar=\{p\in M\mid 0<\deltaS<\smallpar\}$, the following conditions hold:
\begin{itemize}
\item[i)] $\deltaS\colon \Spar\rightarrow [0,\infty)$ is smooth and such that $\|\nabla\deltaS\|_g=1$;
\item[ii)] there exists a diffeomorphism $G\colon (0,\smallpar)\times\{\deltaS=\smallpar\}\rightarrow\Spar$ such that:
\begin{equation}
\deltaS(G(r,p))=r\qquad\text{and}\qquad G_*\partial_r=\nabla\deltaS.
\end{equation}
\end{itemize}
\end{thm}

Theorem \ref{thm:intro_tub_neigh} generalizes the analogous result in \cite[Prop.\ 3.1]{FPR-sing-lapl}, and \cite[Lem.\ 23]{PRS-QC}, to submanifolds of any codimension and is a key tool to build both a canonical probability measure on $S$, cf.\ Lemma \ref{lem:induced_meas}, and the approximation of the relative heat content. Indeed, for any $\eps\leq\smallpar$, denoting by $S_\eps$ the tubular neighborhood of $S$, of radius $\eps$, we consider the (rescaled) relative heat content associated with $S_\eps$, namely
\begin{equation}
\label{eq:intro_approx}
H^{\eps}_S(t)=\frac{1}{\omega(S_\eps)^2}\int_{S_\eps}\int_{S_\eps}p_t(x,y)d\omega(x)d\omega(y),\qquad\forall\,t>0.
\end{equation}
We show that $H^\eps_S$ converges point-wise to $H_S$, as $\eps\to 0$, cf.\ Proposition \ref{prop:hc_approx}. Since now $H^\eps_S$ is the (rescaled) relative heat content associated with a non-characteristic open and bounded set, for any $\eps>0$, we can apply \cite[Thm. 1.1]{ARR-relative}, see also Theorem \ref{t:rel} for the statement, and try to deduce the asymptotics for the limit as $\eps\to 0$. 

Unfortunately, the point-wise convergence of $H^\eps_S$ is too weak to recover any information regarding the small-time asymptotic expansion of the limit. Indeed, by studying an explicit example, we show that the approximation procedure using \eqref{eq:intro_approx} fails to recover the asymptotic expansion of $H_S$. More precisely, we consider a closed simple curve in $\R^3$, equipped with the Euclidean metric and the Lebesgue measure. In this case, on the one hand, by a standard application of the Laplace method, it is possible to compute the asymptotic expansion of $H_S(t)$ as $t\to 0$ at any order: the coefficients appearing in the expansion depend on the curvature of the curve and its derivatives of \emph{any order}. On the other hand, following the approximation strategy described above, we obtain an expansion in the limit which can't possibly agree with the correct one since
\begin{itemize}
\item the only geometrical invariant appearing in the coefficients is the curvature of the curve \emph{without} its derivatives;
\item the orders of the expansion don't agree.
\end{itemize}

\subsection*{Structure of the paper} In Section \ref{sec:prel}, we recall the basic definitions of sub-Rieman\-nian geometry. In Section \ref{sec:tub_submnf}, we give the definition of a non-characteristic submanifold and we prove Theorem \ref{thm:intro_tub_neigh}. In Section \ref{sec:induced_meas}, we build the canonical probability measure on $S$, which is induced by the outer measure $\omega$. In Section \ref{sec:hc_submnf}, we introduce the definition of relative heat content for a submanifold $S$, we build its approximation and we prove its point-wise convergence. Finally, in Section \ref{sec:eucl_ex}, we show how the approximation fails to recover the asymptotic expansion of the relative heat content of $S$.

\subsection*{Acknowledgments.} This work was supported by the Grant ANR-18-CE40-0012 of the ANR, by the Project VINCI 2019 ref. c2-1212 and by ERC Starting Grant 2020, project GEOSUB (grant agreement No. 945655).

\section{Preliminaries}\label{sec:prel}
We recall some essential facts in \sr geometry, following \cite{ABB-srgeom}.

\subsection*{Sub-Riemannian geometry}
Let $M$ be a smooth, connected finite-dimensional manifold. A \sr structure on $M$ is defined by a set of $N$ global smooth vector fields $X_1,\ldots,X_N$, called a \emph{generating frame}. The generating frame defines a \emph{distribution} of subspaces of the tangent spaces at each point $x\in M$, given by
\begin{equation}
\label{eqn:gen_frame}
\D_x=\spn\{X_1(x),\ldots,X_N(x)\}\subseteq T_xM.
\end{equation}
We assume that the distribution is \emph{bracket-generating}, i.e. the Lie algebra of smooth vector fields generated by $X_1,\dots,X_N$, evaluated at the point $x$, coincides with $T_x M$, for all $x\in M$. The generating frame induces a norm on the distribution at $x$, namely
\begin{equation}
\label{eqn:induced_norm}
g_x(v,v)=\inf\left\{\sum_{i=1}^Nu_i^2\mid \sum_{i=1}^Nu_iX_i(x)=v\right\},\qquad\forall\,v\in\D_x,
\end{equation}
which, in turn, defines an inner product on $\D_x$ by polarization. We use the shorthand $\|\cdot\|_g$ for the corresponding norm. We say that $\gamma : [0,T] \to M$ is a \emph{horizontal curve}, if it is absolutely continuous and
\begin{equation}
\dot\gamma(t)\in\D_{\gamma(t)}, \qquad\text{for a.e.}\,t\in [0,T].
\end{equation}
This implies that there exists $u:[0,T]\to\R^N$, such that
\begin{equation}
\dot\gamma(t)=\sum_{i=1}^N u_i(t) X_i(\gamma(t)), \qquad \text{for a.e.}\, t \in [0,T].
\end{equation} 
Moreover, we require that $u\in L^2([0,T],\R^N)$. If $\gamma$ is a horizontal curve, then the map $t\mapsto \|\dot\gamma(t)\|_g$ is integrable on $[0,T]$, see \cite[Lemma 3.12]{ABB-srgeom}. We define the \emph{length} of a horizontal curve as follows:
\begin{equation}
\ell(\gamma) = \int_0^T \|\dot\gamma(t)\|_g dt.
\end{equation}
The \emph{\sr distance} is defined, for any $x,y\in M$, by
\begin{equation}\label{eq:infimo}
d_{\mathrm{SR}}(x,y) = \inf\{\ell(\gamma)\mid \gamma \text{ horizontal curve between $x$ and $y$} \}.
\end{equation}
By Chow-Rashevskii Theorem, the bracket-generating assumption ensures that the distance $d_{\mathrm{SR}}\colon M\times M\to\R$ is finite and continuous. Furthermore it induces the same topology as the manifold one.
\begin{rmk}
The above definition includes all classical constant-rank sub-Riemannian structures as in \cite{montgomerybook,Riffordbook} (where $\D$ is a vector distribution and $g$ a symmetric and positive tensor on $\D$), but also general rank-varying sub-Riemannian structures. The same \sr structure can arise from different generating families. 
\end{rmk}

\subsection*{Geodesics and Hamiltonian flow}


A \emph{geodesic} is a horizontal curve $\gamma :[0,T] \to M$, parametrized with constant speed, and such that any sufficiently short segment is length-minimizing. The \emph{sub-Riemannian Hamiltonian} is the smooth function $H : T^*M \to \R$, 
\begin{equation}
\label{eq:Hamiltonian}
H(\lambda)= \frac{1}{2}\sum_{i=1}^N \langle \lambda, X_i \rangle^2, \qquad \lambda \in T^*M,
\end{equation}
where $X_1,\ldots,X_N$ is a generating frame for the sub-Riemannian structure, and $\langle \lambda, \cdot \rangle $ denotes the action of covectors on vectors. The \emph{Hamiltonian vector field} $\vec H$ on $T^*M$ is then defined by $\varsigma(\cdot,\vec H)=dH$, where $\varsigma\in\Lambda^2(T^*M)$ is the canonical symplectic form.

Solutions $\lambda : [0,T] \to T^*M$ of the \emph{Hamilton equations}
\begin{equation}\label{eq:Hamiltoneqs}
\dot{\lambda}(t) = \vec{H}(\lambda(t)),
\end{equation}
are called \emph{normal extremals}. Their projections $\gamma(t) = \pi(\lambda(t))$ on $M$, where $\pi:T^*M\to M$ is the bundle projection, are locally length-minimizing horizontal curves parametrized with constant speed, and are called \emph{normal geodesics}. If $\gamma$ is a normal geodesic with normal extremal $\lambda$, then its speed is given by $\| \dot\gamma \|_g = \sqrt{2H(\lambda)}$. In particular
\begin{equation}
\label{eq:speed}
\ell(\gamma|_{[0,t]}) = t \sqrt{2H(\lambda(0))},\qquad \forall\, t\in[0,T].
\end{equation} 

There is another class of length-minimizing curves in sub-Riemannian geometry, called \emph{abnormal} or \emph{singular}. As for the normal case, to these curves it corresponds an extremal lift $\lambda(t)$ on $T^*M$, which however may not follow the Hamiltonian dynamics \eqref{eq:Hamiltoneqs}. 
Here we only observe that an abnormal extremal lift $\lambda(t)\in T^*M$ satisfies
\begin{equation}
\label{eq:abn}
\langle \lambda(t),\D_{\pi(\lambda(t))}\rangle=0\quad \text{and} \quad \lambda(t)\neq 0,\qquad \forall\, t\in[0,T] ,
\end{equation}
that is $H(\lambda(t))\equiv 0$. A geodesic may be abnormal and normal at the same time.

\subsection*{Length-minimizers to a submanifold.} Let $S\subset M$ be a closed embedded submani\-fold of codimension $k\geq 0$ and define the \sr distance from $S$:
\begin{equation}
\label{eq:dist_S}
\deltaS(p)=\inf\{d_{\mathrm{SR}}(q,p),q\in S\},\qquad \forall p\in M.
\end{equation}
Let $\gamma:[0,T]\to M$ be a horizontal curve, parametrized with constant speed, such that $\gamma(0)\in S$, $\gamma(T) = p \in M\setminus S$ and assume $\gamma$ is a minimizer for $\deltaS$, that is $\ell(\gamma)=\deltaS(p)$. In particular, $\gamma$ is a geodesic. Any corresponding normal or abnormal lift, say $\lambda :[0,T]\to T^*M$, must satisfy the transversality conditions, cf.\ \cite[Thm 12.13]{AS-GeometricControl},
\begin{equation}\label{eq:trcondition}
\langle \lambda(0), v\rangle=0,\qquad \forall \,v\in T_{\gamma(0)} S.
\end{equation}
Equivalently, the initial covector $\lambda(0)$ must belong to the annihilator bundle $\A(S) = \{\lambda \in T^*M \mid \langle \lambda, T_{\pi(\lambda)} S\rangle = 0\}$ of $S$.
%
%

\subsection*{The heat equation and the relative heat content for a domain}

Let $M$ be a \sr manifold and let $\omega$ be a smooth measure on $M$, i.e. defined by a positive tensor density.  The \emph{divergence} of a smooth vector field is defined by
\begin{equation}
\diverg (X)\omega=\mathcal{L}_X \omega, \qquad\forall\,X\in\Gamma(TM),
\end{equation}
where $\mathcal{L}_X$ denotes the Lie derivative in the direction of $X$. The \emph{horizontal gradient} of a function $f\in C^\infty(M)$, denoted by $\nabla f$, is defined as the horizontal vector field (i.e. tangent to the distribution at each point), such that
\begin{equation}
g_x(\nabla f(x),v)= v(f)(x),\qquad\forall\,v\in\D_x,
\end{equation} 
where $v$ acts as a derivation on $f$. In terms of a generating frame as in \eqref{eqn:gen_frame}, one has
\begin{equation}
\nabla f=\sum_{i=1}^NX_i(f)X_i,\qquad\forall\,f\in C^\infty(M).
\end{equation}

The \emph{sub-Laplacian} is the operator $\Delta= \diverg\circ\nabla$, acting on $C^\infty(M)$. Again, we may write its expression with respect to a generating frame \eqref{eqn:gen_frame}, obtaining
\begin{equation}
\label{eq:def_delta}
\Delta f=\sum_{i=1}^N\left\{X^2_i(f)+X_i(f)\diverg (X_i)\right\},\qquad\forall\,f\in C^\infty(M).
\end{equation}
We denote by $L^2(M,\omega)$, or simply by $L^2$, the space of real functions on $M$ which are square-integrable with respect to the measure $\omega$. Let $\Omega\subset M$ be an open relatively compact set with smooth boundary. This means that the closure $\bar{\Omega}$ is a compact manifold with smooth boundary. We consider the \emph{Cauchy problem for the heat equation} on $\Omega$, that is we look for functions $u$ such that
\begin{equation}\label{eq:cauchy_prob}
\begin{aligned}
\left(\partial_t -\Delta\right)u(t,x)  & =  0, & \qquad &\forall (t,x) \in (0,\infty) \times M, \\
u(0,\cdot) & 	=  \mathds{1}_\Omega,  & \qquad &\text{in }L^2(M,\omega),
\end{aligned}
\end{equation}
where $u(0,\cdot)$ is a shorthand notation for the $L^2$-limit of $u(t,x)$ as $t\to0$. Notice that $\Delta$ is symmetric with respect to the $L^2$-scalar product and negative, moreover, if $(M,d_\mathrm{SR})$ is complete as a metric space, it is essentially self-adjoint, see \cite{MR862049}. Thus, there exists a unique solution to \eqref{eq:cauchy_prob}, and it can be represented as
\begin{equation}
u(t,x)=e^{t\Delta}\mathds{1}_\Omega(x), \qquad\forall\,x\in M,\ t>0,
\end{equation}
where $e^{t\Delta}\colon L^2\rightarrow L^2$ denotes the heat semi-group, associated with $\Delta$. We remark that for all $\varphi \in L^2$, the function $e^{t\Delta} \varphi$ is smooth for all $(t,x) \in (0,\infty) \times M$, by hypoellipticity of the heat operator, see \cite{Hormander}. Furthermore, there exists a heat kernel associated with \eqref{eq:cauchy_prob}, i.e. a positive function $p_t(x,y)\in C^\infty((0,+\infty)\times M\times M)$ such that:
\begin{equation}
\label{eq:hk_def}
u(t,x)=\int_M p_t(x,y)\mathds{1}_{\Omega}(y)d\omega(y)=\int_\Omega p_t(x,y)d\omega(y).
\end{equation}

\begin{defn}[Relative heat content]
Let $u(t,x)$ be the solution to \eqref{eq:cauchy_prob}. We define the \emph{relative heat content}, associated with $\Omega$, as
\begin{equation}
\label{eq:rel_heat}
H_\Omega(t)=\int_\Omega u(t,x)d\omega(x), \qquad \forall\,t>0.
\end{equation}
\end{defn}

In \cite{ARR-relative}, the authors proved the existence of a small-time asymptotic expansion for $H_\Omega(t)$, provided that $\Omega$ is a non-characteristic domain. Precisely, denoting by
\begin{equation}
\delta_{\partial\Omega}(p)=\inf\{d_{\mathrm{SR}}(q,p),q\in \partial\Omega\},\qquad \forall p\in M,
\end{equation}
the distance from the boundary of $\Omega$ and by $\sigma$ the induced \sr measure on $\partial\Omega$ (i.e.\ the one whose density is $\sigma=|i_\nu\omega|_{\partial\Omega}$, where $\nu$ is the outward-pointing normal vector field to $\Omega$), we have the following result. See Definition \ref{def:nonchar_submnf} for the notion of a characteristic point.

\begin{thm}
\label{t:rel}
Let $M$ be a compact \sr manifold, equipped with a smooth measure $\omega$, and let $\Omega\subset M$ be an open subset whose boundary is smooth and has no characteristic points. Then, as $t\to 0$,
\begin{multline}
\label{eq:intro_exp}
H_\Omega(t) =\omega(\Omega)-\frac{1}{\sqrt\pi}\sigma(\partial\Omega)t^{1/2}
\\-\frac{1}{12\sqrt\pi}\int_{\partial\Omega}\left(2g(\nabla\delta_{\partial\Omega},\nabla(\Delta\delta_{\partial\Omega}))-(\Delta\delta_{\partial\Omega})^2\right) d\sigma t^{3/2}+o(t^2).
\end{multline} 
\end{thm}

\begin{rmk}
The compactness assumption in Theorem \ref{t:rel} is technical and can be relaxed by requiring, instead, global doubling of the measure and a global Poincar\'e inequality. We refer to \cite{ARR-relative} for more details.
\end{rmk}


\section{Tubular neighborhood for submanifolds}
\label{sec:tub_submnf}


\begin{defn}[Non-characteristic submanifold]
\label{def:nonchar_submnf}
Let $M$ be a \sr ma\-nifold and let $S\subset M$ be a smooth submanifold of codimension $k\geq 0$. We say that a point $q\in S$ is \emph{non-characteristic} if 
\begin{equation}
\label{eq:nonchar_cond}
\D_q+T_qS=T_qM.
\end{equation}
We say that $S$ is a \emph{non-characteristic submanifold} if \eqref{eq:nonchar_cond} holds for any point $q\in S$.
\end{defn}

\begin{rmk}
Notice that Definition \ref{def:nonchar_submnf} includes the usual one for hypersurfaces, indeed if $S\subset M$ is a submanifold of codimension $1$, it is easy to check that
\begin{equation}
\D_q\subset T_qS\qquad\Leftrightarrow\qquad \D_q+T_qS\subsetneq T_qM.
\end{equation}
\end{rmk}

Under the assumption of non-characteristic submanifold, the distance from $S$, $\deltaS$ defined in \eqref{eq:dist_S}, is smooth and it allows to build smooth tubular neighborhoods of $S$.

\begin{thm}
\label{thm:tub_neigh}
Let $M$ be a \sr manifold and $S\subset M$ be a compact smooth non-characteristic submanifold of codimension $k\geq 1$. Then, there exists $\smallpar>0$ such that, denoting by $\Spar=\{p\in M\mid 0<\deltaS<\smallpar\}$, the following conditions hold:
\begin{itemize}
\item[i)] $\deltaS\colon \Spar\rightarrow [0,\infty)$ is smooth and such that $\|\nabla\deltaS\|_g=1$;
\item[ii)] there exists a diffeomorphism $G\colon (0,\smallpar)\times\{\deltaS=\smallpar\}\rightarrow\Spar$ such that:
\begin{equation}
\deltaS(G(r,p))=r\qquad\text{and}\qquad G_*\partial_r=\nabla\deltaS.
\end{equation}
\end{itemize}
\end{thm}

Before giving the proof of the theorem, we need a preliminary lemma, which can be regarded as a partial generalization of \cite[Prop. 2.7]{FPR-sing-lapl}.

\begin{lem}
\label{lem:abn_char}
Let $M$ be a \sr manifold and $S\subset M$ be a smooth submanifold of codimension $k\geq 1$. Let $\gamma\colon[0,1]\rightarrow M$ be a minimizing geodesic such that 
\begin{equation}
\gamma(0)\in S,\qquad\gamma(1)=p\in M\setminus S,\qquad\deltaS(p)=\ell(\gamma).
\end{equation}
If $\gamma$ is an abnormal geodesic, then $\gamma(0)$ is a characteristic point of $S$.
\end{lem}

\begin{proof}
Let $\lambda\colon [0,1]\rightarrow T^*M$ be an abnormal lift of $\gamma$: this means in particular that $\pi(\lambda(t))=\gamma(t)$ and
\begin{equation}
\label{eq:abn_lift}
\langle \lambda(0),\D_{\gamma(0)}\rangle=0, \qquad\text{with }\lambda(0)\neq 0,
\end{equation}
where $\langle\cdot,\cdot\rangle$ denotes the dual coupling. Moreover, since $\gamma$ is a minimizing geodesic, any lift must necessarily satisfy the transversality condition \eqref{eq:trcondition}. Thus, since $\lambda(0)\neq 0$, conditions \eqref{eq:abn_lift}, \eqref{eq:trcondition} imply that \eqref{eq:nonchar_cond} fails at $q=\gamma(0)$. 
\end{proof}

\begin{proof}[Proof of Theorem \ref{thm:tub_neigh}]
Let us consider the annihilator bundle of $S$, $\A S$, namely the vector bundle of rank $k$, whose fibers are given by 
\begin{equation}
\A_qS=\{\lambda\in T^*_qM\mid\langle\lambda,T_qS\rangle=0\},\qquad\forall q\in S.
\end{equation}
At a point $q\in S$, let us fix a basis of the fiber $\A_qS$, say $\{\lambda_1,\ldots,\lambda_k\}$ and define, for any $j=1,\ldots,k$, the element $v_j\in \D_q$ dual to $\lambda_i$ via the Hamiltonian $H$, i.e.
\begin{equation}
\label{eq:basis_v}
v_j=\pi_*\vec{H}(\lambda_j)=\sum_{i=1}^N\langle \lambda_j,X_i(q)\rangle X_i(q)\qquad j=1,\ldots,k.
\end{equation}
\emph{(Step 1)} If $q$ is non-characteristic, then the set $\{v_1,\ldots,v_k\}$ is linearly independent. 

\noindent Indeed assume there exists constants $\alpha_i$ for $i=1,\ldots,k$, such that $\sum_{i=1}^k\alpha_iv_i=0$. Then, 
\begin{equation}
\label{eq:lin_ind}
0 =\sum_{j=1}^k\alpha_jv_j=\sum_{j=1}^k\alpha_j\sum_{i=1}^N\langle \lambda_j,X_i(q)\rangle X_i(q)=\sum_{i=1}^N\langle\sum_{j=1}^k\alpha_j\lambda_j,X_i(q)\rangle X_i(q)=\pi_*\vec{H}(\lambda),
\end{equation}
having set $\lambda =\sum_{j=1}^k\alpha_j\lambda_j\in\A_qS$. Notice that, by the Lagrange multiplier rule, denoting by $v_\lambda=\pi_*\vec{H}(\lambda)$, for any $\lambda\in T_q^*M$, we have 
\begin{equation}
\label{eq:norm_v}
\|v_\lambda\|^2_g=\inf\left\{\sum_{i=1}^Nu_i^2\mid v_\lambda=\sum_{i=1}^Nu_iX_i(q)\right\}=\sum_{i=1}^N\langle \lambda,X_i(q)\rangle^2=2H(\lambda).
\end{equation}
Therefore, \eqref{eq:lin_ind} implies that $\|\pi_*\vec{H}(\lambda)\|^2_g=2H(\lambda)=0$, or equivalently:
\begin{equation}
\label{eq:abn_lambda}
\langle \lambda,\D_q\rangle= 0.
\end{equation}
Since $\lambda\in\A_qS$ and $q$ is non-characteristic, by \eqref{eq:abn_lambda}, we deduce that $\lambda=0$. Thus:
\begin{equation}
0=\lambda=\sum_{j=1}^k\alpha_j\lambda_j\qquad\Rightarrow\qquad\alpha_j=0,\quad\text{for any }j=1,\ldots,k,
\end{equation}
since $\{\lambda_1,\ldots,\lambda_k\}$ was a basis of the fiber of $\A S$. This concludes the proof of the first step. 
Define now the \sr exponential map from $S$, i.e. the map
\begin{equation}
E\colon D\cap\A S\rightarrow M;\qquad E(\lambda)=\pi\circ e^{\vec H}(\lambda),
\end{equation}
where $D\subset T^*M$ is the open set where the flow of $\vec H$ is defined up to time $1$. Consider also the zero section of the annihilator bundle, namely
\begin{equation}
i\colon S\rightarrow \A S; \qquad i(q)=(q,0)\in \A_qS.
\end{equation}
\emph{(Step 2)} $E$ is a local diffeomorphism at points of $i(S)$.

\noindent To prove the claim, we consider a point $(q,0)\in i(S)$ and verify that $d_{(q,0)}E$ is invertible. Identifying $T_{(q,0)}(D\cap\A S)\cong T_qS\oplus \A_qS$, we have, on the one hand $E\circ i=Id_S$, therefore for a vector $v=(v,0)\in  T_qS\oplus \A_qS$,  
\begin{equation}
d_{(q,0)}E(v)=\left.\frac{d}{dt}\right\rvert_{t=0}E(\lambda(t))=\left.\frac{d}{dt}\right\rvert_{t=0}E\circ i(\gamma(t))=\left.\frac{d}{dt}\right\rvert_{t=0}\gamma(t)=v,
\end{equation}
since $\lambda(t)=(\gamma(t),0)$, with $\gamma\colon(-\eps,\eps)\rightarrow S$, such that $\gamma(0)=q$ and $\dot\gamma(0)=v$. On the other hand, take an element $\lambda=(0,\lambda)\in  T_qS\oplus \A_qS$, then by definition, we obtain
\begin{equation}
d_{(q,0)}E(\lambda)=\left.\frac{d}{dt}\right\rvert_{t=0}E(q,t\lambda)=\left.\frac{d}{dt}\right\rvert_{t=0}E(q,t\lambda)=\left.\frac{d}{dt}\right\rvert_{t=0}\pi\circ e^{t\vec H}(\lambda)=\pi_*\vec H(\lambda)=v_\lambda.
\end{equation}
Thus, choosing any basis for $T_qS$ and the basis $\{\lambda_1,\ldots,\lambda_k\}$ for $\A_qS$, as before, we may write the $n\times n$ matrix representing the differential of $E$ as
\begin{equation}
d_{(q,0)}E=
\begin{pNiceArray}{cc|c}[margin]
\Block{2-2}{Id_{n-k}} & &  \\
&  & v_1,\ldots,v_k\\
\qquad \bm{0} & &
\end{pNiceArray},
\end{equation}
where the vectors $v_j$ are defined in \eqref{eq:basis_v}. Since, by the previous step, the set $\{v_1,\ldots,v_k\}$ is linearly independent in $\D_q$, we conclude that $dE$ is invertible at $i(S)$. 

\noindent\emph{(Step 3)}  There exists $U\subset D\cap \A S$, such that $E\rvert_U$ is a diffeomorphism on its image. Moreover, $U$ can be chosen of the form:
\begin{equation}
\label{eq:setU}
U=\{\lambda\in \A S\mid \sqrt{2H(\lambda)}<\smallpar\}, \qquad\text{for some }\smallpar>0.
\end{equation}
The proof of this step follows verbatim what has been done in \cite[Prop.\ 3.1]{FPR-sing-lapl}, cf.\ also \cite[Lem.\ 23]{PRS-QC}, once we have verified that $\sqrt{2H(\cdot)}$ is a fiber-wise norm on the annihilator bundle. Since $H$ is quadratic on fibers, it immediately follows that $\sqrt{2H(\cdot)}$ is positive, 1-homogeneous and sub-additive. We are left to prove that, for $\lambda\in\A_qS$,
\begin{equation}
\sqrt{2H(\lambda)}=0\qquad\Leftrightarrow\qquad \lambda=0.
\end{equation}
As already remarked in \eqref{eq:abn_lambda}, an element $\lambda\in\A_qS$, such that $\sqrt{2H((q,\lambda))}=0$, annihilates both the distribution and $T_qS$, thus, being $q$ non-characteristic, $\lambda=0$.

\noindent\emph{(Step 4)} $E(U)=\{p\in M\mid \deltaS(p)<\smallpar\}=\Spar\cup S$ and, for elements $(q,\lambda)\in U$ we have $\deltaS(E(q,\lambda))=\sqrt{2H(\lambda)}$. In particular, $\deltaS\in C^\infty(\Spar)$. 

\noindent Firstly, we recall that, for an element $\lambda\in U$, the length of the curve
\begin{equation}
[0,1]\ni t\mapsto\pi\circ e^{t\vec H}(\lambda)\in M
\end{equation}
is equal to $\sqrt{2H(\lambda)}<\smallpar$, as one can check using \eqref{eq:norm_v}. Thus, $E(U)\subset\Spar\cup S$. Secondly, we prove the opposite inclusion: up to restricting $\smallpar$, we may assume that $\Spar\subset K$, for a compact set $K\subset M$. Therefore, for an element $p\in\Spar$, there exists a minimizing geodesic $\gamma\colon[0,1]\rightarrow M$ such that
\begin{equation}
\gamma(0)=q\in S,\qquad\gamma(1)=p\qquad\text{and}\quad\ell(\gamma)=\deltaS(p).
\end{equation}
Applying Lemma \ref{lem:abn_char}, we deduce that $\gamma$ is not an abnormal geodesic, meaning that there exists a unique normal lift for $\gamma$, with initial covector given by $\lambda\in T^*_qM$, which implies 
\begin{equation}
\gamma(t)=\pi\circ e^{t\vec H}(\lambda),
\end{equation}
and in particular, $E(q,\lambda)=p$. Moreover, $\lambda\in U$ as, by optimality, it satisfies the transversality condition \eqref{eq:trcondition}, and also 
\begin{equation}
\label{eq:length_normal}
\ell(\gamma)=\sqrt{2H(\lambda)}<\smallpar, 
\end{equation}
being $p\in\Spar$. Finally, we conclude that $p\in E(U)$ and $\deltaS(E(q,\lambda))=\sqrt{2H(\lambda)}$, by \eqref{eq:length_normal}. Since $\sqrt{2H(\cdot)}$ is smooth, as long as $H(\lambda)\neq 0$, we also have that $\deltaS$ is smooth on the set $E(U\setminus i(S))=\Spar$. 

\noindent\emph{(Step 5)} There exists a diffeomorphism $G\colon (0,\smallpar)\times\{\deltaS=\smallpar\}\rightarrow\Spar$ satisfying item $(ii)$ of the statement. Moreover, $\|\nabla\deltaS\|_g=1$ in $\Spar$. 

\noindent Once again, this part of the proof follows verbatim \cite[Prop.\ 3.1]{FPR-sing-lapl}.
\end{proof}

\begin{rmk}
Consider the set $U=\A S\cap\{\sqrt{2H(\cdot)}<\smallpar\}$ defined in \eqref{eq:setU}. What we proved in the previous Theorem is that $E$ defines a diffeomorphism between $U$ and $\Spar\cup S$. In particular, choosing a local trivialization of the annihilator bundle, this means that 
\begin{equation}
\label{eq:loc_diffeo}
\Spar\cup S\cong\A S\cap\left\{\sqrt{2H(\cdot)}<\smallpar\right\}\underset{locally}{\cong} S\times B^H_{\smallpar}(0),
\end{equation}
where $B^H_{\smallpar}(0)$ denotes the ball of radius $\smallpar$, centered at the origin of the Euclidean space $\left(\R^k,\sqrt{2H(\cdot)}\right)$. Of course, in general, the annihilator bundle will not be globally trivializable, however, this is the case when $S$ is the boundary of an open set and we are able to extend \eqref{eq:loc_diffeo} to the whole submanifold. 
\end{rmk}

Whenever $S$ is a boundary of an open set, we can refine Theorem \ref{thm:tub_neigh} building a double-sided tubular neighborhood of $S$, in which we are able to distinguish the inside and the outside of the open set. This is done using the signed distance function. We recall here its definition.

\begin{defn}[Signed distance]
Let $M$ be a \sr manifold and $\Omega\subset M$ be an open subset. Define $\delta\colon M\rightarrow \R$ to be the \emph{signed distance function} from $\partial\Omega$, i.e.
\begin{equation}
\delta(p)=\begin{cases}
\delta_{\partial\Omega}(p) &p\in \Omega,\\
-\delta_{\partial\Omega}(p) &p\in M\setminus\Omega,  
\end{cases}
\end{equation}
where $\delta_{\partial\Omega}\colon M\rightarrow [0,+\infty)$ denotes the usual distance function from the boundary of $\Omega$. 
\end{defn}

\begin{thm}[Double-sided tubular neighborhood]
\label{thm:double_tub_neigh}
Let $M$ be a \sr manifold and $\Omega\subset M$ be an open, relatively compact subset, whose boundary is smooth and has no characteristic points. Denote by $\Omega_{-\smallpar}^{\smallpar}=\{p\in M\mid -\smallpar<\delta<\smallpar\}$. Then, there exists $\smallpar>0$ such that, the following conditions hold:
\begin{itemize}
\item[i)] $\delta\colon \Omega_{-\smallpar}^{\smallpar}\rightarrow \R$ is smooth and such that $\|\nabla\delta\|_g=1$;
\item[ii)] there exists a diffeomorphism $G\colon (-\smallpar,\smallpar)\times\partial\Omega\rightarrow\Omega_{-\smallpar}^{\smallpar}$ such that:
\begin{equation}
\delta(G(t,p))=t\qquad\text{and}\qquad G_*\partial_t=\nabla\delta.
\end{equation}
\end{itemize}
\end{thm}

\begin{rmk}
The main differences with respect to Theorem \ref{thm:tub_neigh} are that $\delta$ is smooth up to the boundary of $\Omega$, and the diffeomorphism is built starting from $\partial\Omega$. 
\end{rmk}

\begin{proof}[Proof of Theorem \ref{thm:double_tub_neigh}]
By Theorem \ref{thm:tub_neigh}, applied with $S=\partial\Omega$, the \sr exponential map from $\partial\Omega$ is a diffeomorphism for small covectors, namely there exists $\smallpar>0$, such that:
\begin{equation}
E\colon \A(\partial\Omega)\cap\left\{\sqrt{2H(\lambda)}<\smallpar\right\}\overset{\cong}{\longrightarrow}\Omega_{-\smallpar}^{\smallpar}
\end{equation}
and $|\delta(E(q,\lambda))|=\sqrt{2H(\lambda)}$. Now, since $\Omega$ is an open set with smooth boundary, $\A(\partial\Omega)$ is trivializable, i.e. there exists a never-vanishing and inward-pointing smooth section 
\begin{equation}
\lambda^+\colon \partial\Omega\rightarrow\A(\partial\Omega);\qquad q\mapsto\lambda^+_q.
\end{equation}
Furthermore, by non-characteristic assumption, $\sqrt{2H(\cdot)}$ is a fiber-wise norm on the annihilator bundle, hence we may assume without loss of generality that 
\begin{equation}
\sqrt{2H(\lambda^+_q)}=1, \qquad\forall\,q\in S.
\end{equation}
Thus, we find a unique smooth function $\xi(\lambda)\in C^\infty(\A(\partial\Omega))$ such that
\begin{equation}
\lambda=\xi(\lambda)\lambda^+_q,\qquad \lambda\in \A_q(\partial\Omega).
\end{equation}
Hence, the annihilator bundle is trivializable via the map $\xi$, i.e.
\begin{equation}
F\colon\A(\partial\Omega)\overset{\cong}{\longrightarrow} \partial\Omega\times \R;\qquad F(\lambda)=\left(\pi(\lambda),\xi(\lambda)\right).
\end{equation} 
Notice that, by definition, $|\xi(\lambda)|=\sqrt{2H(\lambda)}$. Moreover, $\xi(\lambda)>0$, whenever $E(q,\lambda)\in\Omega$, by definition of $\lambda^+$, $\xi(0)=0$ and negative otherwise. Therefore, having defined the signed distance such that it is positive inside of $\Omega$, we obtain that 
\begin{equation}
\delta(E(q,\lambda))=\xi(\lambda), \qquad\forall\,\lambda\in\left\{\sqrt{2H(\lambda)}<\smallpar\right\},
\end{equation}
proving the smoothness of $\delta$ on the set $\Omega_{-\smallpar}^{\smallpar}$. Finally, define $G$ as the composition of $E\circ F^{-1}$ restricted to the set $(-\smallpar,\smallpar)\times\partial\Omega$. Since $E$ and $F$ are diffeomorphisms, also $G$ is and moreover, 
\begin{equation}
G(t,q)=E(q,t\lambda^+_q)\qquad\forall\,(t,q)\in(-\smallpar,\smallpar)\times\partial\Omega,
\end{equation}
therefore $\delta(G(t,q))=\delta(E(q,t\lambda^+_q))=\xi(t\lambda^+_q)=t$. This concludes the proof.
\end{proof}

\section{Induced measure on \texorpdfstring{$S$}{S}}
\label{sec:induced_meas}

Let $\omega$ be a smooth measure on $M$. We define a measure on $S$ induced by $\omega$, assigning a tensor density. This construction specializes to the \sr perimeter measure, when $S$ is the boundary of an open set. Recall that, by Theorem \ref{thm:tub_neigh}, there exists $\smallpar>0$ such that \eqref{eq:loc_diffeo} holds locally and define $\mathrm{vol}_H$ as the Riemannian measure associated with $\big(\R^k,\|\cdot\|_\perp\big)$, where $\|\cdot\|_\perp$ is a shorthand notation for $\sqrt{2H(\cdot)}\rvert_{\A_qS}$. In particular, $\mathrm{vol}_H$ is well-defined since $\|\cdot\|_\perp$ is induced by the fiber-wise bilinear form
\begin{equation}
\label{eq:metric_fibers}
(\lambda_1,\lambda_2)_\perp=\sum_{i=1}^N\langle \lambda_1,X_i\rangle\langle \lambda_2,X_i\rangle, \qquad\forall\lambda_1,\lambda_2\in\A_qS,\ q\in S,
\end{equation}
where $\langle\cdot,\cdot\rangle$ denotes the dual coupling.

\begin{lem}
\label{lem:induced_meas}
Let $M$ be a \sr manifold and $S\subset M$ be a compact smooth non-characteristic submanifold of codimension $k\geq 1$. Then, there exists a unique smooth probability measure $\muS$ on $S$, such that, 
\begin{equation}
\label{eq:conv_meas}
\int_Mh(p)\omega^\eps(p) \xrightarrow{\eps\to 0} \int_Sh(q)\muS(q),
\end{equation}
for any $h\in C_c(M)$, where, 
\begin{equation}
\omega^\eps=\frac{1}{\omega(S_\eps)}\mathds{1}_{S_\eps},\qquad\forall\,\eps>0,
\end{equation}
\end{lem}

\begin{proof}
Proceeding with hindsight, we are going to define explicitly the measure $\muS$ and then prove the convergence. We may define $\muS$ locally, hence, fix an open coordinate chart $V\subset S$ for $S$ and a local trivialization of $\A S$ over $V$, so that 
\begin{equation}
\A S\rvert_V\cong V\times \R^k.
\end{equation}
By Theorem \ref{thm:tub_neigh}, we have that, denoting by $V_{\smallpar}=E(\A S\rvert_V\cap\{\sqrt{2H(\cdot)}<\smallpar\})$, 
\begin{equation}
V_{\smallpar}\cong\A S\rvert_V\cap\left\{\sqrt{2H(\cdot)}<\smallpar\right\}\cong V\times B^H_{\smallpar}(0).
\end{equation}
Consider on $V_{\smallpar}$, coordinates $(x,z)$ where $(x_1,\ldots,x_{n-k})$ are coordinates on $V$ and $S\cap V_{\smallpar}=\{(x,z)\mid z=0\}$. Thus, since $\omega$ is smooth, we have
\begin{equation}
d\omega(x,z)=\omega(x,z)dxdz,\qquad\text{with }\omega(\cdot)\in C^\infty(V_{\smallpar}),
\end{equation}
where $dx$ and $dz$ are the Lebesgue measures in coordinates. Moreover, since \eqref{eq:metric_fibers} is a metric along the fibers, we can define canonically a volume associated with $H$, which in coordinates is given by 
\begin{equation}
d\mathrm{vol}_H(z)=\sqrt{\det H_q(z)}dz, \qquad\forall\, q\in S,
\end{equation}
with never-vanishing density. Therefore, we may rewrite $\omega$ in terms of $\mathrm{vol}_H$, obtaining
\begin{equation}
\label{eq:omega_vol}
d\omega(x,z)=\omega(x,z)dxdz=\frac{\omega(x,z)}{\sqrt{\det H_q(z)}}dxd\mathrm{vol}_H(z)
\end{equation}
Finally, on the fiber, we can choose an orthonormal (w.r.t. $\sqrt{2H(\cdot)}$) basis of smooth local sections $\{\lambda_1,\ldots,\lambda_k\}$, so that $\mathrm{vol}_H(\lambda_1,\ldots,\lambda_k)=1$, and define $\tilde{\mu}_S$ in coordinates $(x,z)$, to be the contraction of \eqref{eq:omega_vol} along these covectors, restricted to $S$, namely
\begin{equation}
\tilde{\mu}_S=\frac{\omega(x,0)}{\sqrt{\det H_q(0)}}dxd\mathrm{vol}_H(\lambda_1,\ldots,\lambda_k)=\frac{\omega(x,0)}{\sqrt{\det H_q(0)}}dx.
\end{equation}
One can check that this procedure defines a smooth measure on $S$, independently on the choice of the coordinates. We can now verify the convergence, using a partition of the unity argument. Fix a covering of $S$ with a finite number of open charts $\{V_i\}_{i=1}^L$ and consider the associated covering $\{V_{\smallpar}^i\}$ of $S\cup\Spar$, defined by 
\begin{equation}
V_{\smallpar}^i=E\left(\A S\rvert_{V_i}\cap\{\sqrt{2H(\cdot)}<\smallpar\}\right),\qquad\forall\,i=1,\ldots,L.
\end{equation}
Then, consider $\{\rho_i\}_{i=1}^L$ to be a partition of unity subordinate to the covering $\{V_{\smallpar}^i\}$ of $S\cup\Spar$. Exploiting the coordinate expression of $\tilde{\mu}_S$, we have, for any $\eps\leq\smallpar$:
\begin{equation}
\begin{split}
\omega(S_\eps)&=\int_{S_\eps}\sum_{i=1}^L\rho_i(q)d\omega(q)=\sum_{i=1}^L\int_{V_i}\int_{B_\eps^H(x)}\rho_i(x,z)\frac{\omega(x,z)}{\sqrt{\det H_q(z)}}d\mathrm{vol}_H(z)dx\\
							&=\eps^k\sum_{i=1}^L\int_{V_i}\int_0^1\int_{\mathbb{S}^{k-1}}\rho_i(x,\eps r,\theta)\frac{\omega(x,\eps r,\theta)}{\sqrt{\det H_q(\eps r,\theta)}}r^{k-1}d\theta dr dx,
\end{split}
\end{equation}
having expressed the volume $\mathrm{vol}_H$ in polar coordinates $r^{k-1}d\theta dr$. Therefore, up to a factor $\eps^k$, we see that 
\begin{equation}
\frac{\omega(S_\eps)}{\eps^k}\xrightarrow{\eps\to 0}\varpi_k\sum_{i=1}^L\int_{V_i}\rho_i(x,0) d\tilde{\mu}_S(x)=\varpi_k\int_Sd\tilde{\mu}_S
\end{equation}
where $\varpi_k$ is the volume of the standard unit ball in $\R^k$. Finally, reasoning as above, since for any $h\in C_c(M)$, we are able to extract a factor $\eps^k$ from the integral of $h$ over $S_\eps$, we obtain the convergence in the weak-star topology \eqref{eq:conv_meas}, having normalized $\tilde{\mu}_S$ to obtain a probability measure $\muS$. 
\end{proof}

\begin{rmk}
Since $\omega^\eps$, for any $\eps\leq\smallpar$, has compact support which is contained in $\Spar$, we can extend the convergence \eqref{eq:conv_meas} to any continuous function on $M$. 
\end{rmk}

\section{Heat content for submanifolds}
\label{sec:hc_submnf}

Let $M$ be a \sr manifold, equipped with a smooth measure $\omega$, and let $S\subset M$ be a smooth, compact submanifold of codimension $k\geq 1$. We may consider $\mu$ a smooth probability measure on $S$ as initial datum for the heat equation, in the sense of distributions, and study the associated Cauchy problem:
\begin{equation}\label{eq:he_submnf}
\begin{aligned}
(\partial_t -\Delta)u(t,x)  =  0, & &\qquad &\forall (t,x) \in (0,\infty) \times M, \\
u(t,\cdot) \xrightarrow{t\to0}  \mu, & &  \qquad & \text{in }\mathcal{D}'(M).
\end{aligned}
\end{equation}
A solution to this problem, in the sense of distribution, is given by 
\begin{equation}
u(t,x)=\int_Sp_t(x,y)d\mu(x), \qquad\forall\, (t,x) \in (0,\infty) \times M,
\end{equation}
which, by hypoellipticity, is a smooth function for positive times. Recall that, by definition of the relative heat content associated with an open set $\Omega\subset M$, we have
\begin{equation}
\label{eq:rhc_aux}
H_\Omega(t)=\int_\Omega\int_{\Omega}p_t(x,y)d\omega(x)d\omega(y),\qquad\forall\,t>0.
\end{equation} 
Thus, a suitable generalization of \eqref{eq:rhc_aux} for a submanifold $S$ seems to be:
\begin{equation}
\label{eq:hc_S}
H_S(t)=\int_S\int_Sp_t(x,y)d\mu(x)d\mu(y),\qquad\forall\,t>0.
\end{equation}
Moreover, when $S$ is non-characteristic, Lemma \ref{lem:induced_meas} provides with a canonical probability measure on $S$, induced by $\omega$, i.e. $\muS$. Henceforth, we assume $S$ non-characteristic and fix $\mu=\muS$. In this setting, we can hope to obtain an asymptotic expansion of \eqref{eq:hc_S}. 

\begin{prop}
\label{prop:hc_approx}
Let $M$ be a \sr manifold, equipped with a smooth measure $\omega$, let $S\subset M$ be a smooth, compact and non-characteristic submanifold of codimension $k\geq 1$ and fix the probability measure $\muS$ on $S$. Define, for any $\eps\leq\smallpar$, 
\begin{equation}
\label{eq:def_approx}
H^{\eps}_S(t)=\int_M\int_Mp_t(x,y)d\omega^\eps(x)d\omega^\eps(y),\qquad\forall\,t>0.
\end{equation}
Then, for any $t>0$, 
\begin{equation}
\label{eq:point_conv}
H^{\eps}_S(t)\xrightarrow{\eps\to0}H_S(t).
\end{equation}
\end{prop}

\begin{proof}
Firstly, notice that, applying Lemma \ref{lem:induced_meas}, we have that
\begin{equation}
\label{eq:conv_sol}
u^\eps(t,x)=\int_Mp_t(x,y)d\omega^\eps(y)=\langle\omega^\eps,p_t(x,\cdot)\rangle\xrightarrow{\eps\to 0}\langle\muS,p_t(x,\cdot)\rangle,
\end{equation}
for any $t>0$ and $y\in M$. Secondly, since the heat kernel $p_t$ is smooth on $M\times M$, there exists a constant $C(t)>0$ depending on $t$, such that:
\begin{equation}
\label{eq:hk_est}
\left\|p_t(\cdot,\cdot)\right\|_{L_{\mathrm{loc}}^\infty(M\times M)}\leq C(t),
\end{equation}
and  we remark that the constant $C(t)$ explodes as $t\to0$. Therefore, the convergence \eqref{eq:conv_sol} is locally uniform with respect to $x\in M$. In conclusion, 
\begin{equation}
\begin{split}
|H^{\eps}_S(t)-H_S(t)|&=|\langle \omega^\eps,u^\eps(t,\cdot)\rangle-\langle \muS,u(t,\cdot)\rangle|\\
											&\leq|\langle \omega^\eps,u^\eps(t,\cdot)-u(t,\cdot)\rangle|+|\langle \omega^\eps,u(t,\cdot)\rangle-\langle \muS,u(t,\cdot)\rangle|\\
											&\leq \|u^\eps(t,\cdot)-u(t,\cdot)\|_{L^\infty(\Spar)}|\langle \omega^\eps,1\rangle|+|\langle \omega^\eps-\muS,u(t,\cdot)\rangle|,
\end{split}
\end{equation}
and taking the limit as $\eps\to 0$ in the last line proves the desired result. 
\end{proof}

\begin{rmk}
The convergence \eqref{eq:point_conv} is never uniform as $t\to 0$, being the constant $C(t)$ in \eqref{eq:hk_est} not bounded as $t\to 0$. This suggests that, while $H_S^\eps$ seems to be the best possible approximation of the heat content associated with $S$, using such a strategy to deduce the asymptotics of $H_S(t)$ is not correct. Indeed, we can show that the coefficients of the expansion of $H_S^\eps$ can not approximate those of $H_S(t)$, in general.
\end{rmk}

\section{An example: closed simple curve in \texorpdfstring{$\R^3$}{R3}}
\label{sec:eucl_ex}
In $\R^3$ equipped with the Euclidean scalar product and the Lebesgue measure, let us consider a biregular closed simple curve, parametrized by arc-length, $\gamma\colon[0,\ell]\rightarrow\R^3$, where $\ell$ denotes the length of $\gamma$. Recall that a smooth curve $\gamma\colon I\rightarrow\R^3$ is biregular if 
\begin{equation}
\dot\gamma(s)\wedge\ddot\gamma(s)\neq 0, \qquad\forall\,s\in I,
\end{equation}
where $\wedge$ denotes the cross product in $\R^3$. In this setting, define $S=\gamma([0,\ell])\subset\R^3$, submanifold of codimension $2$. The tubular neighborhood of $S$ given by Theorem \ref{thm:tub_neigh} coincides with the usual Euclidean tubular neighborhood, which can be conveniently described by the Frenet-Serret moving frame along $\gamma$, it being biregular. In particular, denoting by $\{T(s),N(s),B(s)\}$ the Frenet-Serret frame for $s\in[0,\ell]$, we have,
\begin{equation}
S_\eps=\{\gamma(s)+r(\cos\theta N(s)+\sin\theta B(s))\mid s\in[0,\ell],\ \theta\in (0,2\pi],\ r\in (0,\eps)\}, \qquad\forall\eps\leq\smallpar.
\end{equation}
Thus, in coordinates $(s,r,\theta)$, the Lebesgue measure is 
\begin{equation}
\label{eq:leb_coord}
dxdydz=r(1-rk(s)\cos\theta)dsdrd\theta, 
\end{equation}
where $k(s)=\|\ddot\gamma(s)\|$ is the curvature of $\gamma$, and the procedure of Lemma \ref{lem:induced_meas} gives the probability measure $\muS=ds/\ell$. Following the discussion of Section \ref{sec:hc_submnf}, we define the heat content associated with $S$ as
\begin{equation}
H_S(t)=\frac{1}{\ell^2}\frac{1}{(4\pi t)^{3/2}}\int_0^\ell\int_0^\ell e^{-\frac{|\gamma(s)-\gamma(\tau)|^2}{4t}}dsd\tau, \qquad\forall t>0.
\end{equation}

\subsection*{Asymptotic expansion of \texorpdfstring{$H_S(t)$}{HS(t)}}

Denote by $\phi_\tau(s)=|\gamma(s)-\gamma(\tau)|^2$. We can explicitly compute the asymptotic expansion of $H_S$ applying the Laplace method to the integral
\begin{equation}
\label{eq:lapl_int}
I_\tau(\lambda)=\int_0^\ell e^{-\phi_\tau(s)\lambda}ds, \qquad \text{as }\lambda\to +\infty.
\end{equation}
In particular, since $\gamma$ is a simple curve, the phase $\phi_\tau(s)$ has a strict minimum at $s=\tau$, thus there exists $\eps=\eps(\tau)$ such that $\phi_\tau'(s)>0$ for any $s\in[\tau-\eps,\tau+\eps]\setminus\{\tau\}$. A direct computation, building upon $\|\dot\gamma\|=1$, yields:
\begin{equation}
\begin{aligned}
&\phi_\tau'(\tau)=0,  &\qquad & \phi_\tau''(\tau)=2, 						 	 &\qquad & \phi_\tau^{(5)}(\tau)=-5\partial_\tau\left(k(\tau)^2\right),\\
&\phi_\tau'''(\tau)=0,&\qquad & \phi_\tau^{(4)}(\tau)=-2k(\tau)^2, &\qquad & \phi_\tau^{(6)}(\tau)=-9\partial_\tau^2\left(k(\tau)^2\right)+2\|\dddot\gamma(\tau)\|^2.
\end{aligned}
\end{equation}
Therefore the phase has a Taylor expansion at its minimum and we can apply the Laplace method, which gives a full asymptotic expansion, cf. \ \cite[Thm. 8.1]{MR1429619},
\begin{equation}
I_\tau(\lambda)\sim e^{-\lambda \phi_\tau(\tau)}\sum_{i=0}^\infty\Gamma\left(\frac{i+1}{2}\right)\frac{a_i(\tau)}{\lambda^{\frac{i+1}{2}}}\qquad\text{as }\lambda\to\infty,
\end{equation}
where $\Gamma$ is the Euler Gamma function, and the $a_i(\tau)$ are given by explicit formulas in terms of the derivatives of $\phi_\tau$ at its minimum. Moreover, since, for any $\tau\in[0,\ell]$, the phase has an interior minimum, the odd coefficients of the expansion vanish. For the even-order coefficients, we have
\begin{equation}
\label{eq:coeff_form}
a_0=1,  \qquad  a_2=\frac{1}{8}k(\tau)^2, \qquad  a_4=\frac{1}{1152}\left(36\partial_\tau^2\left(k(\tau)^2\right)+35k(\tau)^2-8\|\dddot\gamma(\tau)\|^2 \right).
\end{equation} 
To conclude, we have to integrate with respect to $\tau$ the asymptotic expansion of \eqref{eq:lapl_int}. In general, the expansion may not be uniform in $\tau\in (0,\ell)$, however, since $\gamma$ is uniformly continuous on $[0,\ell]$ and can be extended by periodicity on the whole real line, the choice of $\eps>0$ such that $\phi_\tau'(s)>0$ for any $s\in[\tau-\eps,\tau+\eps]\setminus\{\tau\}$ can be made uniform, providing uniform estimates of the remainder. In particular, we have
\begin{equation}
H_S(t)\sim\frac{1}{\ell^2}\frac{1}{(4\pi t)^{3/2}}\int_0^\ell\int_{\tau-\eps}^{\tau+\eps} e^{-\frac{\phi_\tau(s)}{4t}}dsd\tau, \qquad\text{as }t\to 0,
\end{equation} 
where $\eps>0$ is chosen uniformly with respect to $\tau$. Hence, we conclude that:
\begin{equation}
\label{eq:hc_exp}
\begin{split}
H_S(t)&\sim \frac{1}{\ell^2}\frac{1}{(4\pi t)^{3/2}}\int_0^\ell\sum_{i=0}^\infty\Gamma\left(\frac{2i+1}{2}\right)a_{2i}(\tau)(4t)^{\frac{2i+1}{2}}d\tau\\
			&=\frac{1}{\ell^2}\frac{1}{4\pi t}\sum_{i=0}^\infty\alpha_it^i,
\end{split}
\end{equation}
as $t\to 0$, where the coefficients $\alpha_i$'s are defined by 
\begin{equation}
\alpha_i=2^{2i-1}(2i+1)\int_0^\ell a_{2i}(\tau)d\tau,\qquad\forall\,i\geq0, 
\end{equation}
and are given explicitly by \eqref{eq:coeff_form}, up to $i=2$.

\subsection*{Approximation via tubular neighborhoods}
We compute the asymptotic expansion of the approximation defined in Proposition \ref{prop:hc_approx}, when $S=\gamma([0,\ell])$, and we compare its coefficients with those obtained in \eqref{eq:hc_exp}. Recall that, by \eqref{eq:def_approx}, we set
\begin{equation}
H^{\eps}_S(t)=\frac{1}{|S_\eps|^2}\int_{S_\eps}\int_{S_\eps}\frac{1}{(4\pi t)^{3/2}} e^{-\frac{|x-y|^2}{4t}}dxdy,\qquad\forall\,t>0.
\end{equation} 
By Theorem \ref{t:rel}, there exists an asymptotic expansion up to order $4$ in $\sqrt{t}$, of the form
\begin{equation}
\label{eq:hca_eps}
H^{\eps}_S(t) =\frac{1}{|S_\eps|^2}\left(\alpha_0^\eps+\alpha_1^\eps t^{1/2}+\alpha_3^\eps t^{3/2}+o(t^2)\right),
\end{equation}
where $\alpha_0^\eps=|S_\eps|$ and 
\begin{equation}
\label{eq:coeff_eps}
\alpha_1^\eps=-\frac{1}{\sqrt\pi}\sigma^\eps(\partial S_\eps),\qquad\alpha_3^\eps=-\frac{1}{12\sqrt\pi}\int_{\partial S_\eps}\left(2\nabla\deltaS\cdot\nabla(\Delta\deltaS )-(\Delta\deltaS)^2 \right)d\sigma^\eps.
\end{equation}
Notice that, in tubular coordinates $(s,r,\theta)$, $\nabla\deltaS=\partial_r$, and, since \eqref{eq:leb_coord} holds,  
\begin{equation}
d\sigma^r=r(1-rk(s)\cos\theta)dsd\theta,\qquad\Delta\deltaS= \frac{1}{r}+\partial_r\left(\log(1-rk(s)\cos\theta)\right). 
\end{equation}
Thus, we can explicitly compute the coefficients \eqref{eq:coeff_eps}:
\begin{equation}
\label{eq:coeff_eps2}
\begin{split}
\alpha_1^\eps &=-\frac{\eps}{\sqrt\pi}\int_0^\ell\int_0^{2\pi}(1-\eps k(s)\cos\theta)dsd\theta=-2\eps\ell\sqrt\pi\\
\alpha_3^\eps &=-\frac{\eps}{12\sqrt\pi}\int_0^\ell\int_0^{2\pi}\left(-\frac{3}{\eps^2}+\frac{A_1(s,\eps,\theta)}{\eps}+A_0(s,\eps,\theta)\right)(1-\eps k(s)\cos\theta)dsd\theta, 
\end{split}
\end{equation}
where $A_0,A_1$ are smooth functions defined by 
\begin{equation}
\begin{split}
A_0(s,r,\theta) &=2\partial^2_r\left(\log(1-rk(s)\cos\theta)\right)-(\partial_r\left(\log(1-rk(s)\cos\theta)\right))^2,\\
A_1(s,r,\theta) &=-2\partial_r\left(\log(1-rk(s)\cos\theta)\right).
\end{split}
\end{equation}

\subsection*{Comparison between the two approaches}
Let us compare the asymptotics of the two quantities in exam: for a fixed $\eps>0$ and as $t\to 0$, we have from \eqref{eq:hc_exp} and \eqref{eq:hca_eps}
\begin{equation}
\begin{split}
H^{\eps}_S(t) &= \frac{1}{|S_\eps|^2}\left(\alpha_0^\eps+\alpha_1^\eps t^{1/2}+\alpha_3^\eps t^{3/2}+o(t^2)\right),\\
       H_S(t) &= \frac{1}{\ell^2}\frac{1}{4\pi t}\left(\alpha_0+\alpha_1 t+\alpha_3 t^2+o(t^2)\right),
\end{split}
\end{equation}
where the coefficients are given by \eqref{eq:coeff_eps2} and \eqref{eq:coeff_form}, respectively. At this stage, on the one hand, we notice that the order in $t$ of the expansions doesn't agree. On the other hand, the coefficients $\alpha_1^\eps,\alpha_3^\eps$ do not contain fine geometrical information of $S$, indeed, the functions $A_0$, $A_1$ depend only the curvature of $\gamma$, as opposed to \eqref{eq:coeff_form}, where derivatives of $k(s)$ appear. Moreover, at a formal level, $\alpha_1^\eps\to 0$ as $\eps\to 0$, whereas $\alpha_3^\eps$ explodes: it is possible to give meaning to these limits, taking into account the parabolic scaling between the space and time variables of the heat equation and formally replacing $t\mapsto \eps^2 t$, in \eqref{eq:hca_eps}, however, we do not recover any geometrical meaning. 

\begin{rmk}
The approximating relative heat content is too coarse a tool to detect the geometry of a submanifold of high codimension and a different strategy is needed. Inspired by the construction of the tubular neighborhood for $S$, we may study the asymptotic expansion of $H_S(t)$ using a perturbative approach similarly to what has been done in \cite{MR3207127}, presenting the sub-Laplacian of $S$ as a perturbation of a simpler operator. This will be object of future research. 
\end{rmk}
%

\bibliographystyle{alphaabbr}
\bibliography{biblio-submnf-TSG}

\begin{thebibliography}{CdVHT20}

\bibitem[ABB20]{ABB-srgeom}
A.~Agrachev, D.~Barilari, and U.~Boscain.
\newblock {\em A comprehensive introduction to sub-{R}iemannian geometry},
  volume 181 of {\em Cambridge Studies in Advanced Mathematics}.
\newblock Cambridge University Press, Cambridge, 2020.

\bibitem[ARR21]{ARR-relative}
A.~Agrachev, L.~Rizzi, and T.~Rossi.
\newblock Relative heat content asymptotics for sub-riemannian manifolds.
\newblock {\em arXiv preprint arXiv:2110.03926}, 2021.

\bibitem[AS04]{AS-GeometricControl}
A.~Agrachev and Y.~L. Sachkov.
\newblock {\em Control theory from the geometric viewpoint}, volume~87 of {\em
  Encyclopaedia of Mathematical Sciences}.
\newblock Springer-Verlag, Berlin, 2004.
\newblock Control Theory and Optimization, II.

\bibitem[Bar13]{MR3207127}
D.~Barilari.
\newblock Trace heat kernel asymptotics in 3{D} contact sub-{R}iemannian
  geometry.
\newblock {\em J. Math. Sci. (N.Y.)}, 195(3):391--411, 2013.
\newblock Translation of Sovrem. Mat. Prilozh. No. 82 (2012).

\bibitem[CdVHT20]{YHT-2}
Y.~Colin~de Verdi{\`e}re, L.~Hillairet, and E.~Tr{\'e}lat.
\newblock Small-time asymptotics of hypoelliptic heat kernels near the
  diagonal, nilpotentization and related results.
\newblock {\em arXiv preprint arXiv:2004.06461}, 2020.

\bibitem[DG94]{D-G}
S.~Desjardins and P.~Gilkey.
\newblock Heat content asymptotics for operators of {L}aplace type with
  {N}eumann boundary conditions.
\newblock {\em Math. Z.}, 215(2):251--268, 1994.

\bibitem[FPR20]{FPR-sing-lapl}
V.~Franceschi, D.~Prandi, and L.~Rizzi.
\newblock On the essential self-adjointness of singular sub-{L}aplacians.
\newblock {\em Potential Anal.}, 53(1):89--112, 2020.

\bibitem[H{\"{o}}r67]{Hormander}
L.~H{\"{o}}rmander.
\newblock Hypoelliptic second order differential equations.
\newblock {\em Acta Math.}, 119:147--171, 1967.

\bibitem[Mon02]{montgomerybook}
R.~Montgomery.
\newblock {\em A tour of subriemannian geometries, their geodesics and
  applications}, volume~91 of {\em Mathematical Surveys and Monographs}.
\newblock American Mathematical Society, Providence, RI, 2002.

\bibitem[Olv97]{MR1429619}
F.~W.~J. Olver.
\newblock {\em Asymptotics and special functions}.
\newblock AKP Classics. A K Peters, Ltd., Wellesley, MA, 1997.
\newblock Reprint of the 1974 original [Academic Press, New York; MR0435697 (55
  \#8655)].

\bibitem[PRS18]{PRS-QC}
D.~Prandi, L.~Rizzi, and M.~Seri.
\newblock Quantum confinement on non-complete {R}iemannian manifolds.
\newblock {\em J. Spectr. Theory}, 8(4):1221--1280, 2018.

\bibitem[Rif14]{Riffordbook}
L.~Rifford.
\newblock {\em Sub-{R}iemannian geometry and optimal transport}.
\newblock SpringerBriefs in Mathematics. Springer, Cham, 2014.

\bibitem[RR21]{MR4223354}
L.~Rizzi and T.~Rossi.
\newblock Heat content asymptotics for sub-{R}iemannian manifolds.
\newblock {\em J. Math. Pures Appl. (9)}, 148:267--307, 2021.

\bibitem[Sav98]{Savo-heat-cont-asymp}
A.~Savo.
\newblock Uniform estimates and the whole asymptotic series of the heat content
  on manifolds.
\newblock {\em Geom. Dedicata}, 73(2):181--214, 1998.

\bibitem[Str86]{MR862049}
R.~S. Strichartz.
\newblock Sub-{R}iemannian geometry.
\newblock {\em J. Differential Geom.}, 24(2):221--263, 1986.

\bibitem[TW18]{TW-heat-cont-hei}
J.~Tyson and J.~Wang.
\newblock Heat content and horizontal mean curvature on the {H}eisenberg group.
\newblock {\em Comm. Partial Differential Equations}, 43(3):467--505, 2018.

\bibitem[vdBDG93]{vdB-D-G}
M.~van~den Berg, S.~Desjardins, and P.~Gilkey.
\newblock Functorality and heat content asymptotics for operators of {L}aplace
  type.
\newblock {\em Topol. Methods Nonlinear Anal.}, 2(1):147--162, 1993.

\bibitem[vdBG94]{vdB-G1}
M.~van~den Berg and P.~B. Gilkey.
\newblock Heat content asymptotics of a {R}iemannian manifold with boundary.
\newblock {\em J. Funct. Anal.}, 120(1):48--71, 1994.

\bibitem[vdBG15]{MR3358065}
M.~van~den Berg and P.~Gilkey.
\newblock Heat flow out of a compact manifold.
\newblock {\em J. Geom. Anal.}, 25(3):1576--1601, 2015.

\bibitem[vdBLG94]{vdB-LG}
M.~van~den Berg and J.-F. Le~Gall.
\newblock Mean curvature and the heat equation.
\newblock {\em Math. Z.}, 215(3):437--464, 1994.

\end{thebibliography}

\end{document}